\documentclass[a4paper, 12pt]{article}
\usepackage{fullpage}
\usepackage{amsmath, amsthm, amssymb, mathrsfs, graphicx, subfigure}
\usepackage{ifthen, url}
\usepackage[usenames]{color}
\usepackage{hyperref}
\usepackage{times}
\usepackage{bm}
\usepackage{epstopdf}

\usepackage{natbib}

\theoremstyle{plain}
\newtheorem{theorem}{Theorem}

\newtheorem{lemma}{Lemma}

\theoremstyle{remark}
\newtheorem{remark}{Remark}

\theoremstyle{definition}

\theoremstyle{remark}

\def\newtheta{\tilde\vartheta_n}
\def\rawcf{\lambda}

\def\bx{\mathbf{x}}
\def\abs#1{\vert#1\vert}

\def\var{\text{var}}

\begin{document}

\title{On the validity of the formal Edgeworth expansion for posterior densities}

\author{
John E. Kolassa \\
Department of Statistics and Biostatistics \\
Rutgers University \\
\url{kolassa@stat.rutgers.edu} \\
\mbox{} \\
Todd A. Kuffner \\
Department of Mathematics \\
Washington University in St. Louis \\
\url{kuffner@wustl.edu} \\
}

\maketitle

\begin{abstract}
We consider a fundamental open problem in parametric Bayesian theory, namely the validity of the formal Edgeworth expansion of the posterior density. While the study of valid asymptotic expansions for posterior distributions constitutes a rich literature, the validity of the formal Edgeworth expansion has not been rigorously established. Several authors have claimed connections of various posterior expansions with the classical Edgeworth expansion, or have simply assumed its validity. Our main result settles this open problem. We also prove a lemma concerning the order of posterior cumulants which is of independent interest in Bayesian parametric theory. The most relevant literature is synthesized and compared to the newly-derived Edgeworth expansions. Numerical investigations illustrate that our expansion has the behavior expected of an Edgeworth expansion, and that it has better performance than the other existing expansion which was previously claimed to be of Edgeworth-type.
\medskip

\end{abstract}

\section{Introduction}
\label{sec:intro}

The Edgeworth series expansion of a density function is a fundamental tool in classical asymptotic theory for parametric inference. Such expansions are natural refinements to first-order asymptotic Gaussian approximations to large-sample distributions of suitably centered and normalized functionals of sequences of random variables, $X_{1}, \ldots, X_{n}$. Here, $n$ is the available sample size, asymptotic means $n \rightarrow \infty$, and first-order means that the approximation using only the standard Gaussian distribution incurs an absolute approximation error of order $O(n^{-1/2})$. The term formal in conjunction with Edgeworth expansions means that derivation of the expansion begins by expanding the log characteristic function, and then utilizes Fourier inversion to obtain the corresponding density \citep[p. 280]{BNCox:1979}. The coefficients of such expansions are expressed in terms of cumulants of the underlying density, together with a set of orthogonal basis functions for a suitably general hypothesized function space for the density being approximated. Asymptotic expansions are said to be valid if the absolute approximation error incurred, as an order of magnitude in $n$, by truncating the series expansion after a finite number of terms, is of the same asymptotic order as the first omitted term. The validity of the formal Edgeworth expansion is of foundational importance, in the sense that this implies a certain degree of regularity of the statistical model, and the expansion itself offers deeper insights into the finite sample performance of many frequentist inference procedures, such as those based on the likelihood. Applying standard arguments to justify term-by-term integration of the truncated Edgeworth expansion for a density yields the corresponding Edgeworth approximation for the cumulative distribution function. Such expansions are essential to studying the coverage accuracy of confidence sets, as well as establishing higher-order relationships between different methods for constructing such approximate confidence sets. Our understanding of the bootstrap has also been greatly enhanced by studying connections with Edgeworth expansions \citep{Hall:1992}. 

In contrast to central limit theorems in the frequentist context, the large-sample Gaussian approximation of the posterior distribution of a suitably centered and normalized parameter is typically justified using a Bernstein-von Mises theorem. Such theorems establish stochastic convergence of the total variation distance between the sequence of posterior distributions and an appropriate Gaussian distribution, where stochastic convergence is with respect to the true distribution from which samples are independently drawn. The exact form depends on the centering statistic and its corresponding variance estimate. For a lucid discussion, see \citet[Ch. 10]{vanderVaart:1998}.

Somewhat surprisingly, the validity of formal Edgeworth expansion for the posterior density, arising from a Bayesian analysis, has not previously been established. This is less surprising when the challenging nature of this problem is understood. The term valid in the context of posterior expansions means that, when approximating the posterior by truncating the series, the absolute error is uniformly of the proper order on a set of parameter values whose posterior probability does not go to zero. While some authors have studied related expansions, or made claims about the similarity of such expansions to classical Edgeworth expansions, to our knowledge there is no existing proof of the validity of the formal Edgeworth series expansion for posterior distributions. Apart from formal Edgeworth expansion validity being of foundational importance, we note that approximate posterior inference through higher-order asymptotics remains of interest in parametric Bayesian theory; see, for example, \citet{RSV:2014} or \citet{KharroubiSweeting:2016}. Other approximate Bayesian inference procedures, such as variational Bayes, or approximate Bayesian computation, have become popular due to their ability to circumvent computationally expensive Markov chain Monte Carlo procedures. Higher-order asymptotics offers another route to approximate Bayesian inference which, in many settings of practical interest, can be extremely accurate and inexpensive to implement.

Within the existing literature on posterior expansions, there have been two dominant approaches to obtaining approximations of posterior quantities. The common starting point is to express the posterior mean or density as a ratio of two integrals. Expansions for the numerator and denominator separately, which may be truncated and integrated to yield integral approximations, can yield valid expansions for the posterior quantity through formal division of the numerator and denominator expansions. One approach which heavily emphasizes Taylor expansion is found in \citet{Johnson:1967,Johnson:1970} and \citet{GSJ:1982}.  The most popular approach, however, is to apply Laplace's method to approximate the respective integrals, and then use the ratio of these approximations; see \citet{Lindley:1961,Lindley:1980,Davison:1986,TierneyKadane:1986}. Validity of Laplace expansions for posterior densities is considered in \citet{KTK:1990}.   

Another well-established method of posterior expansions utilizes Stein's identity;  see, for instance, \citet{Woodroofe:1989,Woodroofe:1992,Weng:2003,WengTsai:2008} and \citet{Weng:2010}. 

In this paradigm, \citet{Weng:2010} claimed to have established an Edgeworth expansion for the posterior density. Compared to other expansions, Weng's approach most closely resembles the final form of an Edgeworth expansion in that it is expressed in terms of moments, but it is not a formal Edgeworth expansion, and its structure is actually quite different from an Edgeworth expansion, as we show below.

\citet{BickelGhosh:1990}, in a paper establishing Bartlett correctability of the posterior distribution of the likelihood ratio statistic, implied that their regularity conditions would imply existence of Edgeworth expansions for the posterior of a particular functional, centered and normalized by maximum likelihood quantities. However, they never actually claimed that posterior Edgeworth expansions had been established as valid.

In fact, we will explain below that careful examination of their regularity conditions shows that they simply assume the validity of a posterior Edgeworth expansion in order to establish the validity of the posterior Bartlett correction. They do not actually prove that the Edgeworth expansion is valid.

To establish validity of the formal Edgeworth expansion for a density, it is required to show that the coefficients in the expansion, i.e. the cumulants of the statistical functional which is being approximated, are of the proper asymptotic order to ensure that the terms of the expansion have the claimed orders as powers of $n^{-1/2}$. This entails formally proving that power series expansions for those cumulants are valid. 

The existing results concerning validity of cumulant expansions are all within the sampling distribution framework, and as such are not applicable to the Bayesian setting. Numerous authors have studied expansions for posterior moments, but such results do not actually imply that the corresponding cumulants are of the proper asymptotic order. 

Edgeworth expansion relies on proper order for cumulants of the variable under investigation, after dividing by its standard deviation.  These cumulants for the standardized variable are known as invariant cumulants, and demonstrating their proper order is more delicate than demonstrating the proper order for the underlying moments.  As an example, consider the relationship between the posterior variance $\sigma^2$ and the fourth central moment $\mu'_4$.  In order for the formal Edgeworth expansion to have the proper asymptotic behavior, $\mu'_4\sigma^{-4}-3=O(1/n)$, and so $\mu'_4=\sigma^4 O(1/n)$. However, the converse -- that $\mu'_4=\sigma^4 O(1/n)$ implies $\mu'_4\sigma^{-4}-3=O(1/n)$ -- does not hold.  Hence bounds on the moments (even for central moments) cannot guarantee proper size of the invariant cumulants.  Results exploring the parameter centered at something only approximating the posterior mean (for example, the maximum likelihood estimator), and standardized by something other than the exact posterior standard deviation (for example, an approximation based on the Fisher information) will not ensure proper order for the invariant cumulants.

In this paper, we make several novel contributions. First, we prove the validity of the formal Edgeworth series expansion of the posterior density and distribution function. This requires us to prove a lemma concerning the asymptotic order of posterior cumulants, which is of independent interest, and appears to be the first rigorously established general result of this type. We also synthesize the relevant literature on posterior expansions, giving rigorous explanations of how existing Edgeworth-type expansions are not actually formal Edgeworth expansions. Finally, we provide a numerical illustration of our results. 

\section{Background}

\subsection{Posterior Expansions}

Before proceeding, we note that one could consider either analytic or stochastic expansions for posterior densities. For an analytic expansion, the observed data sequence is viewed as a subsequence of a given, fixed infinite sequence of realizations. Deriving analytic expansions amounts to showing that the posterior has certain asymptotic properties for a given, well-behaved infinite sequence of observations. The stochastic expansion viewpoint asserts that such well-behaved sequences occur with probability tending to one, with respect to the true data generating probability distribution. In this paper, we consider analytic approximations for a given  well-behaved infinite sequence of observations, though it would be possible to give analogous stochastic versions where $O(\cdot)$ terms are replaced by corresponding $O_{p}(\cdot)$ terms; see \citet{Sweeting:1995a} and \citet{KTK:1990}.

From a formal Edgeworth series perspective, existing posterior expansions are centered at the wrong place, typically either the maximum likelihood estimator or true parameter value, instead of the posterior mean or something which is approximating the posterior mean. As noted by \citet[\S 20.8]{DasGupta:2008}, the maximum likelihood estimator and posterior mean are closely related. Suppose that one observes a sequence of observations $X^{(n)}=(X_{1}, \ldots, X_{n})$, each of which are identical copies of a random variable $X$ whose distribution $P_{\theta}$ depends on a scalar parameter $\theta$. Write $P_{\theta_{0}}$ for the distribution corresponding to the true value $\theta_{0}$ under which each component of $X^{(n)}$ is generated. Let $E(\theta | X^{(n)})$ denote the posterior mean under some prior density, and let $\hat{\theta}_{n}$ be the maximum likelihood estimator for $\theta$ based on $X^{(n)}$. Under standard regularity conditions, $E(\theta | X^{(n)})-\hat{\theta}_{n}$, and also $n^{1/2}(E(\theta | X^{(n)})-\hat{\theta}_{n})$ converge in $P_{\theta_{0}}$-probability to zero. Therefore, 
\begin{equation}\label{mlepostmean}
E(\theta | X^{(n)})-\hat{\theta}_{n}=o_{p}(n^{-1/2})
\end{equation}
under $P_{\theta_{0}}$. The effects of centering in the wrong place are examined in \S~\ref{sec:comparison}. In particular, we discuss the expansions given by \citet{Weng:2010} and \citet{Hartigan:1965}. These two expansions are not Edgeworth series expansions, but for reasons explained in \S~\ref{sec:comparison}, these can be considered to have the closest relationship to our formal Edgeworth expansions. 

It may appear strange to the reader that we are claiming the Edgeworth expansion for the posterior has not been established as valid, even under regularity conditions common in the literature. After all, there is the celebrated posterior Bartlett correction of \citet{BickelGhosh:1990}, and the conventional derivation of the validity of the Bartlett correction requires a valid Edgeworth expansion. Some authors, e.g. \citet{ChangMukerjee:2006}, refer to the Bickel and Gosh regularity conditions on the Bayesian model specification \citep[p. 1078]{BickelGhosh:1990} as Edgeworth assumptions. Indeed, \citet{BickelGhosh:1990} simply assume that an Edgeworth expansion exists, but do not prove that it is valid. In particular, condition $B_{m}(iv)$ in \citet[p. 1079]{BickelGhosh:1990} is made to ensure that, when approximating the mean of the loglikelihood statistic by the coefficient on the $n^{-1}$ term in the expansion of that mean, the remainder term in the approximation is of a small enough order. A similar assumption is condition (i) of \citet[p. 755]{BGV:1985}. Such assumptions are tantamount to assuming that the Edgeworth expansion for the posterior is valid, which in turn implies that the cumulant expansions are assumed to be valid. Therefore, the validity of the Edgeworth expansion is assumed in these papers, but not formally proven. Moreover, the methods of derivation and resulting terms of these expansions disqualify them from consideration as formal Edgeworth expansions. 

Related to the approach of \citet{BickelGhosh:1990}, some authors use an expansion for the loglikelihood at the maximum likelihood estimate to obtain a posterior expansion which has some structure resembling an Edgeworth expansion; see \citet[Eq. 2.2.19]{DattaMukerjee:2004}. As with other expansions mentioned above, this is not an Edgeworth expansion for several reasons. First, it is not derived by formal expansion of the posterior characteristic function. Second, the centering is at the maximum likelihood estimate, not the posterior mean. Third, the first correction term is a linear one, which vanishes in an Edgeworth expansion. Moreover, the coefficients are not cumulants of the posterior. We further note that in \citet{BickelGhosh:1990} and \citet[Lemma 4.2.1]{DattaMukerjee:2004}, an approximation is given for the posterior characteristic function loglikelihood ratio statistic. \citet{DiCiccioStern:1993} consider approximation of the posterior moment generating function of this statistic. All of these expansions rely on regularity conditions which amount to assuming the validity of the Edgeworth expansion, though none of these papers contain proofs, nor do they actually use formal Edgeworth expansions in their arguments. A lucid discussion of frequentist and Bayesian Bartlett correction, and where the assumption of validity of Edgeworth expansions is essential, is found in \citet{DiCiccioStern:1994}.

\subsection{Validity and Formal Edgeworth Expansions for Sampling Distributions}
In the frequentist context, asymptotic expansions and their components are studied with respect to the sampling distribution of the relevant statistical functional, under hypothetical repeated sampling. \citet{Wallace:1958} provided the conventional notion of validity for an asymptotic expansion.  Suppose each function $g_{n}(y)$ in a sequence $\{g_{n}\}_{n \geq 1}$ is approximated by any partial sum of a series $\sum_{j=0}^{\infty}n^{-j/2} A_{j}(y)$, where the $A_{j}(\cdot)$ do not depend on $n$. If for some constant $C_{r}(y)$, the absolute errors satisfy
\begin{equation}
\nonumber
\Big| g_{n}(y)-\sum_{j=0}^{r}n^{-j/2}A_{j}(y) \Big| \leq n^{-(r+1)}C_{r}(y),
\end{equation}
then the asymptotic expansion is said to be valid to $r$ terms. If the constant $C_{r}(y)$ does not depend on $y$, then the asymptotic expansion is called uniformly valid in $y$. Hence, validity requires that the absolute error in the approximation, using any partial sum, is of the same order of magnitude as the first neglected term. 

Consider a scalar random variable $X$ with characteristic function $\gamma_{X}(t)=E \exp(itX)$, and denote by $X^{(n)}=(X_{1}, \ldots, X_{n})$ a sequence of independent and identically distributed copies of $X$. Suppose that it is required to approximate the density $g_{n}(y)$ of $Y_{n}=n^{1/2} s(X_{1}, \ldots, X_{n})$ for some scalar-valued function $s(\cdot)$, such that $Y_{n}$ is a centered and scaled statistic possessing an asymptotically standard normal distribution to first order. The formal Edgeworth expansion of the density $g_{n}(y)$ is derived according to the following steps; see, e.g. \citet[\S 1.5]{Jensen:1995}, \citet[Ch. 5]{McCullagh:1987}, \citet[Ch. 2]{Hall:1992}, or \citet[Ch. 3]{Kolassa:2006}. First, Taylor expand the cumulant generating function of $Y_{n}$, $\log \gamma_{Y_{n}}(t)$, in a neighborhood of zero, $|t|<cn^{1/2}$ for some $c>0$. Next, expand the Fourier inversion integral over the region $|t|<cn^{1/2}$. Then, obtain a bound on the inversion integral over the region $|t|>cn^{1/2}$. If $Y_{n}$ satisfies the assumptions of the smooth function model \citep[\S 2.4]{Hall:1992}, then one can follow the program in the references above to rigorously establish the validity of the Edgeworth expansion for $g_{n}(y)$. Other standard references for Edgeworth series expansions include \cite{Feller:1971}, \citet{BhattacharyaGhosh:1978, BhattacharyaRao:2010} and \citet{Ghosh:1994}.

To ensure that the formal Edgeworth series expansion for the density of $Y_{n}$ is valid in the sense of \citet{Wallace:1958}, it is required that the $j$th cumulant of $Y_{n}$, denoted by $\kappa_{j,n}$, is of order $n^{-(j-2)/2}$, and may expanded in a power series in $n^{-1}$:
\begin{equation} \label{eqn:cumexp}
\kappa_{j,n}=n^{-(j-2)/2}(c_{j,1}+n^{-1}c_{j,2}+n^{-2}c_{j,3}+\cdots), \quad j \geq 1 . 
\end{equation}
Since $Y_{n}$ is centered and scaled so that $\kappa_{1,n}=E(Y_{n}) \rightarrow 0$ and $\kappa_{2,n}=\var(Y_{n}) \rightarrow 1$, then $c_{1,1}=0$ and $c_{2,1}=1$. The origins of this result in the frequentist, repeated sampling setting can be traced to the combinatorial arguments of \citet{James:1955,James:1958}, \citet{JamesMayne:1962} and \citet{LeonovShiryaev:1959}.  The interested reader is referred to \citet{Withers:1982, Withers:1984}, \citet[Chapter 2]{McCullagh:1987}, \citet[Chapter 2]{Hall:1992}, \citet{Mykland:1999}, \citet{Kolassa:2006}, and \citet[Chapters 12 and 13]{StuartOrd:1994} for more details about cumulant expansions.

It is also of interest to integrate the Edgeworth series expansion of the density of $Y_{n}$ to obtain an Edgeworth expansion for its corresponding distribution function. 

Analogously to the density setting, this expansion is desired to be valid for fixed $j$ as $n \rightarrow \infty$, and the remainder should be of the stated order uniformly in $y$. Sufficient regularity conditions \citep[\S 2.2]{Hall:1992} for the validity of this expansion to order $j$ are that $E(|X|^{j+2})<\infty$ and
\begin{equation} \label{eqn:Cramer}
\limsup_{|t| \rightarrow \infty} |\gamma_{X}(t)|<1 .
\end{equation}
The latter condition is known as Cram\'{e}r's condition. 

\subsection{Conditional Expansions and Posterior Expansions}

In the frequentist sampling distribution framework concerning expansions for conditional densities, one might approach the problem by writing the conditional density as the ratio of a joint density to a marginal density. After deriving Edgeworth expansions for the numerator and denominator, formal division of these series expansions yields what is referred to as a direct-direct Edgeworth expansion for the conditional density \citep[Ch. 7]{BNCox:1989}. Proving validity for these direct-direct expansions requires the analogous proofs of validity for expansions of conditional cumulants, which are in general very difficult. Such direct-direct expansions are not the same as a direct expansion of a conditional density, but the bigger obstacle to their utility is that they are non-Bayesian in nature. Standard sampling distribution arguments do not apply when deriving an Edgeworth expansion for the posterior density of $\vartheta_{n}=n^{1/2} (\theta-\theta_0)/ \sigma$, where $\theta_0$ and $\sigma$ are the posterior expectation and standard deviation. In particular, in this posterior setting, one does not have independent and identically distributed $\theta$s, but rather a single $\theta$. Furthermore, posterior inference is conditional on a single data set, without appealing to repeated sampling arguments. An obvious point worth emphasizing is that the consideration of increasingly larger sample sizes is not the same as considering hypothetical repeated sampling. As discussed above, we are assuming that the data represent a subsequence from a fixed infinite sequence, rather than repeated random samples from a probability distribution.

A major obstacle to proving validity of posterior Edgeworth expansions is the issue of the cumulant orders. In the sampling distribution framework, there are well-known results concerning the relationship between conditional and unconditional cumulants, but these results are unfortunately of no use in the posterior framework. In particular, \citet{Brillinger:1969} established a theorem which permits computation of unconditional cumulants from conditional cumulants; see also \citet{Speed:1983} and \citet[\S 2.9 and \S 5.6]{McCullagh:1987}. However, there is no converse to Brillinger's theorem, and even if there were, one must still overcome the issue that $\theta$ is a single random variable, rather than a sequence of independent and identically distributed random variables.

Due to the challenges just mentioned, there are no general results about posterior cumulants available in the literature. \citet{PSS:1993} give some specific results regarding the form of the cumulant generating function only relevant to exponential families. \citet[p. 1145]{Hartigan:1965} alluded to the order of posterior cumulants, but was not precise about how such orders could be shown.

\section{Main results}
We consider expansions for the posterior distribution of the scalar-valued quantity $\vartheta_{n}=n^{1/2}(\theta-\theta_0)/\sigma$, for $\theta_0$ and $\sigma$ the posterior expectation and standard deviation. It is assumed throughout that an appropriate Bernstein-von Mises theorem holds for the sequence of posterior distributions of $\vartheta_{n}$. Due to the large variety of asymptotic normality results in the literature, and since our goal is to prove validity of the Edgeworth expansion in some generality, we do not discuss all of the conditions needed for the various specific Bernstein-von Mises theorems to hold. Our treatment of regularity conditions focuses on those conditions of particular relevance to the validity of the formal Edgeworth expansion. 

The first step in our analysis is to prove that the invariant posterior cumulants admit a valid power series expansion, establishing that the coefficients in the Edgeworth series expansion will have the correct asymptotic order. We then prove validity of the expansion for the posterior density and distribution function, respectively.

\subsection{The Order of Posterior Cumulants}
Given a sequence $X^{(n)}=(X_{1}, \ldots, X_{n})$ of $n$ random variables which, conditional on the value of a scalar random parameter $\theta$ taking values in a set $\Theta \subset \Re$, are independent and identically distributed according to density $f(x| \theta)$, define $L_{n}(\theta ; \bx)=\prod_{i=1}^{n}f(x_{i} | \theta)$ to be the likelihood function. Here, $\bx$ represents the observations of the sequence $X^{(n)}$. Denote the loglikelihood function by $\ell_{n}(\theta)\equiv \ell_{n}(\theta ; \bx)$. Assume that, prior to observing the data, the uncertainty about $\theta$ is described by a prior density function $\pi(\theta)$. The posterior density of $\theta$ is defined as
\[
f(\theta|\bx)= \frac{\pi(\theta)\prod_{i=1}^{n}f(x_{i}|\theta)}{\int_{\Theta}\pi(\theta)\prod_{i=1}^{n}f(x_{i}|\theta) d\theta} = \frac{L_{n}(\theta ; \bx)\pi(\theta)}{\int_{\Theta}L_{n}(\theta ; \bx)\pi(\theta)d\theta} .
\]
Throughout, we suppress the dependence on $\bx$ when there is no chance of confusion. 

\begin{lemma}\label{lem:cumulants}
Assume that the likelihood function has a single global maximizer $\hat\theta_n$.
Define the average loglikelihood $\bar\ell_n(\theta)=\ell(\theta)/n$, and assume that $\bar\ell_n$ and the log prior density have six continuous derivatives in a neighborhood of the form $\hat\theta_n\pm\epsilon$ for $\epsilon$ independent of $n$, and such that $\bar \ell_{n}(\theta)<\bar\ell_{n}(\hat\theta_n)-\delta$ for $\theta\notin(\hat\theta_n-\epsilon,\hat\theta_n+\epsilon)$, and assume that the second derivative of the average loglikelihood is bounded away from zero on this neighborhood.  Then the invariant cumulant of $\theta$ of order $j$, defined to be the cumulant of order $j$ of $(\theta-\theta_0)/\sigma$, and denoted by $\kappa_{j}$, is $O(n^{(2-j)/2})$ for $j\in\{3,4,5\}$.  Here $\theta_0$ and $\sigma$ are the expectation and standard deviation of the posterior distribution, respectively.
\end{lemma}

\begin{proof}
Arguments in the first half of this proof hold for loglikelihood functions and log prior densities with varying numbers of derivatives; denote this number by $k$, and until specified otherwise, it might, but need not be, 5.
By (\ref{mlepostmean}), there exists $N$ (potentially dependent on the sample) so that for $n\geq N$, $|\theta_0-\hat\theta_n|<\epsilon/2$, and so continuous derivatives to order $k$ exist for $\bar \ell_{n}(\theta)$ and the log prior density at $\theta_0$. 
Hence the log prior has an expansion 
\begin{equation*}
-\sum_{j=0}^k h_j(\theta-\theta_0)^j/j!+Q(\theta)(\theta-\theta_0)^{k+1}/(k+1)!,
\end{equation*}
and the loglikelihood has an expansion
\begin{equation*}
-n[\sum_{j=0}^k g_j(\theta-\theta_0)^j/j!+Q^*(\theta)(\theta-\theta_0)^{k+1}/(k+1)!],
\end{equation*}
where the coefficients $g_j$ and $h_j$ may be calculated from derivatives of the loglikelihood and log prior, respectively. 
Here $Q(\theta)$ and $Q^*(\theta)$ are the standard Taylor series remainder terms, calculated from the 
derivatives of order $k+1$ of the log prior density and average loglikelihood, respectively, evaluated at a parameter value intermediate between $\theta_0$ and $\theta$. 

The log posterior can be expressed as 
\begin{equation*}
\sum_{j=0}^k p_j(\theta-\theta_0)^j/j!+[
n Q^*(\theta)+Q(\theta)],
\end{equation*}
where $p_j=-h_j-n g_j$.
The first term $p_0$ may be chosen to make the posterior integrate to $1$.
The error terms $Q(\theta)$ and $Q^*(\theta)$ may be taken as bounded for $\theta\in(\hat\theta_n-\epsilon/2,\hat\theta_n+\epsilon/2)$.
The choice of $\theta_0$ ensures that $p_1=O(n^{1/2})$. 

Let $\omega(\theta)$ be the polynomial resulting from retaining only terms of order $k$ and smaller in the power series for $\exp(\sum_{j=3}^k p_j(\theta-\theta_0)^j)$.
Let $\mu^*$ represent the extended Laplace approximation to the posterior 
moments;
that is, 
\begin{equation*}\mu^*_j=\int_{-\infty}^\infty
\exp(-p_0)\exp(-\frac{1}{2} (\theta-\theta_0) ^2 p_2)\theta^j \omega(\theta) ~d\theta.
\end{equation*}

Let $\mu_j$ denote the true values of these moments.
Choose $\epsilon>0$ to satisfy the conditions of the lemma such that both
\begin{equation*}
|\log(\pi(\theta_0))+n\bar \ell_{n}(\theta_0)-\log(\pi(\theta))-n\bar \ell_{n}(\theta)|\leq p_2(\theta-\theta_0)^2/4,
\end{equation*}
and 
\begin{equation*}
|\sum_{j=3}^k p_j(\theta-\theta_0)^j|\leq p_2(\theta-\theta_0)^2/2 ,
\end{equation*}
for $|\theta-\theta_0|<\epsilon$. Since the difference between the maximum likelihood estimator and the posterior expectation is $O_p(1/n^{1/2})$, then there exists $\delta>0$ such that $\bar \ell_{n}(\theta)<\bar \ell_{n}(\theta_0)-\delta$ for $\theta\in(-\epsilon/2,\epsilon/2)^c$, and
the contribution from outside of the interval to the absolute approximation error $|\mu_j^*-\mu_j|$ is bounded by $\exp(-n\delta)$.
Inside $(-\epsilon/2,\epsilon/2)$, the contribution to 
$\abs{\mu_j-\mu^*_j}$ is bounded by
\begin{equation*}
\exp(-p_0)\exp(-n\frac{1}{4} (\theta-\theta_0) ^2 p_2)\left(|n Q^*(\theta)+Q(\theta)|+\left(\sum_{j=3}^k p_j(\theta-\theta_0)^j\right)^{k+1}/(k+1)!\right),
\end{equation*} 
by \citet[Theorem 2.5.3]{Kolassa:2006}, which is $O(n^{-(k+1)/2})$.
    
Take $k=5$.
In this case,
\begin{eqnarray*}
\omega(\theta)&=&1 - p_1(\theta-\theta_0)
+p_2^2(\theta-\theta_0)^2/2
-(p_1^3+p_3)(\theta-\theta_0)^3/6\\&&
+(p_1^4+4 p_3 p_1-p_4)(\theta-\theta_0)^4/24\\&&
-(\theta-\theta_0) ^5 (p_1 ^5+10 p_3
    p_1 ^2-5 p_4  p_1 +p_5 )/120.\end{eqnarray*}
Moments approximated in this way are accurate to $O(n^{-7/2})$, and cumulants approximated using
standard formulas for producing cumulants from moments \citep[p. 10]{Kolassa:2006} are accurate to the same order. Denote the cumulant of $\theta$ of order $j$ by $\beta_{j}$. The first cumulant is
$\beta_1=g_1 g_2^{-1}+(g_1 h_2 g_2^{-2}- h_1g_2^{-1})n^{-1}+O(1/n^2)$. Recall that the $g_j$ terms are the coefficients in the expansion of the average loglikelihood about the posterior mean. They are all $O(1)$ except $g_{1}$. Since the posterior mean is within $O(n^{-1/2})$ of the maximum likelihood estimate, then the choice of $\theta_0$ as the posterior mean forces $g_1=O(n^{-1/2})$.
The second cumulant is  $\beta_2=\frac{1}{2} g_2^{-1}n^{-1} +O(n^{-2})$, the third cumulant is
$\beta_3=-\frac{1}{6}g_3 g_2^{-3}n^{-2}+O(n^{-3})$, and
the fourth cumulant is $\beta_4=-\frac{1}{24}(2 g_3^2 g_2^{-5}-g_4 g_2^{-4})n^{-3}+O(n^{-7/2})$. 
The most delicate part of the argument is the calculation of these cumulants, and ensuring that larger terms cancel to leave a remainder of the proper order.  The proof is then completed by dividing the cumulants by the proper power of the second cumulant, $\beta_2$, which is $O(n^{-1})$, to see that the quotient is of the proper order. Specifically, the invariant cumulant of order $j$, $j\geq 3$, is $\kappa_{j}=\beta_{j}/\beta_{2}^{j/2}$. Since $g_2, g_3, g_4$ are bounded away from zero, the invariant cumulants of order $3$ and $4$ are $O(n^{-1/2})$ and $O(n^{-1})$ respectively.  Similar calculations show that the invariant cumulant of order $5$ is $O(n^{-3/2})$. 
\end{proof}

\begin{remark}
The argument above holds to provide bounds on moments of all orders.  Orders of cumulants are more delicate, since cumulants are expressed in terms of differences of products of moments, and proper order for cumulants requires that leading terms of the representation properly cancel.  At present we know of no way to do this except on a case-by-case basis.  This difficulty extends to bounds on invariant cumulants.
\end{remark}
\begin{remark}
In \S~\ref{sec:intro} we argued that one cannot bound the cumulants based on the order of the moments.  We are instead bounding cumulants by getting a Laplace expansion for the moments, and observing that enough leading terms cancel to show that the cumulants are of the proper order.
\end{remark}
\begin{remark}
As one would expect, the $h_j$ terms, corresponding to coefficients in the expansion of the log prior about the posterior mean, appear only in terms of order $n^{-1}$ and smaller.
\end{remark}

\subsection{Validity of Edgeworth expansion for the posterior density}
\begin{theorem}\label{thm:main}
Let $\cal A$ be a subset of the sample space.  Suppose that for any $\bx\in\cal{A}$, the following assumptions hold.
\begin{enumerate}
\item The prior is absolutely continuous with respect to Lebesgue measure.\label{c1}
\item The likelihood function $L_n(\theta ; \bx)$ is a measurable function of $\theta$.\label{c2}
\item The posterior is proper, and, for sufficiently large $n$, has a bounded density.\label{c3}
\item The likelihood function has a unique global maximizer $\hat\theta_n(\bx)$,\label{c4}
\item The loglikelihood $\ell_n(\theta)$ is $k$-times differentiable in a neighborhood of $\hat\theta_{n}(\bx)$, with average second derivative $\ell''_n(\theta(\bx),\bx)/n$ bounded away from zero, and average $j$th derivative $\ell^{(j)}_n(\theta(\bx),\bx)/n$ bounded, for $2\leq j\leq k$.\label{c5}
\end{enumerate}
Define $\vartheta_{n}=n^{1/2}(\theta-\theta_{0})/\sigma$, with $\theta_{0}$ the posterior mean and $\sigma$ the standard deviation of the posterior distribution for $\theta$. Let 
\begin{equation}\label{esdef}
e_{k,n}(\vartheta)=\phi(\vartheta)\{1+h_3(\vartheta)\kappa_3/6+\kappa_4 h_4(\vartheta)/24-\kappa_3^2 h_6(\vartheta)/72+\cdots\},
\end{equation} 
where $\kappa_j$ are cumulants of $\vartheta_{n}$, and hence the invariant cumulants of $\theta$, and $e_{k,n}$ is truncated to contain only terms with products of $\kappa_{j}$ of the form $\prod_{m=1}^{j}\kappa_{r_{m}}$, such that $\sum_{m=1}^{j}(r_{m}-2)<k$. Heuristically, $e_{k,n}$ contains terms of size larger than $O(n^{-k/2})$, where $\kappa_{j}$ satisfies
\begin{equation}
\kappa_{j}=O(n^{-(j-2)/2}), \quad j \geq 3 .
\end{equation}
Then the error in the use of $e_{k,n}(\vartheta)$ to approximate the posterior density is of order $O(n^{-k/2})$, uniformly in $\vartheta$ and uniformly in $\bx$ in a compact subset of the sample space, but not relatively.
\end{theorem}
\begin{proof}
Let
$\rawcf(\tau)=\int_{-\infty}^\infty L_n(\theta)\pi(\theta)\exp(i\theta\tau)~d\theta$.
The characteristic function for $\theta$ is then
$\rawcf(\tau)/\rawcf(0),$ and
the characteristic function for $\vartheta_{n}=n^{1/2}(\theta-\theta_0)/\sigma$ is
\begin{equation}
\varphi(\tau)=\rawcf(n^{1/2}\tau)\exp(-n^{1/2}\theta_0\tau i)/\rawcf(0).
\end{equation}
Here $\sigma$ is the standard deviation for the posterior distribution for $\theta$.

The Riemann--Lebesgue theorem, using Assumption~\ref{c3} above, indicates that 
$\abs{\rawcf(\tau)} \leq C/\abs{\tau}$ for
$C=2\int_{-\infty}^\infty L_n(\theta)\pi(\theta)~d\theta$; see \citet[Theorem 26.1]{Billingsley:1995}.
Furthermore, by Assumption~\ref{c3}, there exist $m$ and $C_1$ and $C_2$ such that 
\begin{equation}\label{tailsmall}
\int_{-\infty}^\infty\abs{\rawcf(\tau)}^m~d\tau\leq C_1 \qquad \hbox{ and hence } \qquad \abs{\varphi(\tau)}\leq
C_2/(n^{1/2}\abs{\tau}) . \end{equation}
Then, in parallel with the development of \citet[\S XV.3]{Feller:1971}, the Fourier inversion of $\varphi$ to obtain $e_{k,n}(\vartheta)$ as the posterior density of $\vartheta_{n}$ is performed by first expanding $\phi$ in $\theta$ near $\hat\theta$.
The Fourier inversion integral results approximately in $e_{k,n}$. 

More formally,
let $\gamma_n(\zeta)=\int_{-\infty}^\infty\exp(i\zeta\vartheta) e_{k,n}(\vartheta)~d\zeta$, for $e_{k,n}(\vartheta)$ defined in (\ref{esdef}).
In parallel with the development of \citet[\S XVI.2]{Feller:1971}, 
the posterior density for $\vartheta_{n}$ is given by 
\begin{equation}
\frac{1}{2\pi}\int_{-\infty}^\infty\exp(-i\zeta\vartheta)
\varphi(\zeta)~d\zeta,
\end{equation}
and the difference between the true posterior density and the Edgeworth series approximation of (\ref{esdef}) is bounded by
\begin{eqnarray}
 \frac{1}{2\pi}\int_{-\infty}^\infty\exp(-i\zeta\vartheta)|\varphi(\zeta)&-&\gamma_n(\zeta)|~d\zeta\nonumber\\
&=& \frac{1}{2\pi}\int_{(-\delta n^{1/2},\delta n^{1/2})}\exp(-i\zeta\vartheta)|\varphi(\zeta)-\gamma_n(\zeta)|~d\zeta\nonumber\\
&\;\;\;+& \frac{1}{2\pi}\int_{(-\delta n^{1/2},\delta n^{1/2})^c}\exp(-i\zeta\vartheta)|\varphi(\zeta)-\gamma_n(\zeta)|~d\zeta . \label{intbnd}
\end{eqnarray}
As discussed by \citet[\S 3.7]{Kolassa:2006}, the first of these integrals is bounded by 
\begin{equation*}
\frac{1}{2\pi}\int_{-\infty}^\infty\exp(-i\zeta\vartheta)p(\zeta,n)/n^{k/2}~d\zeta,
\end{equation*}
for $p(\zeta,n)$ a polynomial in $\zeta$ and $1/n^{1/2}$.  This polynomial has coefficients that depend on derivatives of the loglikelihood, and so the error is uniformly of the proper order.  
Note that result (\ref{tailsmall}) applies to $\gamma_n(\zeta)$ as well as to $\varphi(\zeta)$; choose the resulting constants $m^*\geq m$, $C_1^*\geq C_1$, and $C_2^*\geq C_2$.  These together show that the second integral in (\ref{tailsmall}) is bounded by 
\begin{equation*}
(C_2^*/(\delta n^{1/2}))^{n-m^*}C_1^* n/\delta ,
\end{equation*} 
which is geometrically small.
\end{proof}

\subsection{Validity of Edgeworth expansion for the posterior distribution function}
By assuming that the prior is a density, and that the likelihood is continuously differentiable, we actually have more smoothness than is required for Cram\'{e}r's condition \eqref{eqn:Cramer} to hold. However, as noted above, the extra smoothness implied by these assumptions is necessary to prove the validity of the cumulant expansions in Lemma~\ref{lem:cumulants}.
\begin{theorem}
Under the assumptions of Theorem~\ref{thm:main}, define the Edgeworth approximation to the posterior cumulative distribution function of $\vartheta$ as 
\begin{equation}\label{escdfdef}
E_{k,n}(\vartheta)=\Phi(\vartheta)-\phi(\vartheta)\{h_2(\vartheta)\kappa_3/6+\kappa_4 h_3(\vartheta)/24-\kappa_3^2 h_5(\vartheta)/72+\cdots\} .
\end{equation} 
The absolute error incurred in using \eqref{escdfdef} to approximate the distribution function of $\vartheta_{n}$ is uniformly of order $O(n^{-3/2})$
\end{theorem}
\begin{proof}
As noted by \citet[Equation (48)]{Kolassa:2006}, the error in applying (\ref{escdfdef}) to approximate the posterior distribution function is given by the left side of  (\ref{intbnd}), modified by dividing the integrand by $|\zeta|$, in this case with a density.  By the previous application of the Riemann-Lebesgue theorem, the modified integral representing error converges absolutely, and is multiplied by the proper power of the sample size.
\end{proof}

\section{Relationship to Existing Edgeworth-type Expansions}
\label{sec:comparison}

\citet{Weng:2010} provides an asymptotic expansion for the posterior distribution of a statistical model for data consisting of $n$ independent and identically distributed observations, satisfying certain regularity conditions, by centering the distribution at the maximum likelihood estimate, and scaling the difference between a potential parameter value and the estimate by the second derivative of the loglikelihood evaluated at the maximum likelihood.  This produces an asymptotic expansion valid to $O(n^{-(s+1)/2})$, and uses $3 s-1$ terms.  For example, when approximating the posterior CDF, the approximation with error $O(n^{-3/2})$ uses $s=2$, and hence uses five terms, including the leading term represented by the normal cumulative distribution function.  This expansion includes Hermite polynomials to order 5, as is found in the standard Edgeworth expansion presented by, for example, \citet{McCullagh:1987}.  However, since the maximum likelihood estimate is not the same as the posterior expectation, the leading term in the Weng approximation does not match the target distribution as well as one centered at the true posterior expectation.  This lack of match leads to a more complicated expansion.  Furthermore, the example Weng presents provides finite sample performance that is inferior to that generally expected from an approximation with asymptotic error $O(n^{-3/2})$, as we illustrate in \S~\ref{sec:example}.

We now demonstrate that Weng's approximation has an error that is equivalent to that of our approximation; these calculations also demonstrate the differences between the two approximatons.
Weng's approximation for the posterior distribution function of $\tilde{\vartheta}_{n}=(\theta-\hat{\theta}_n)/\hat{\sigma}$, where $\hat{\theta}_n$ is the maximum likelihood estimate and $\hat{\sigma}$ is the square root of the observed information evaluated at $\hat{\theta}_n$, is of the form 
\begin{equation*}
P[(\theta-\hat\theta)/\hat\sigma<\newtheta | \bx]=
\Phi(\newtheta)-\sum_{i=1}^{3 s}q_{i-1}(\newtheta)\phi(\newtheta) c_i,\end{equation*}
for $q_i$ Hermite polynomials, and $c_i$ constants given by Stein's lemma.
\citet{Weng:2010} shows that this approximation holds uniformly for $\tilde{\vartheta}\in \Re$.

\citet{Hartigan:1965} also provides an approximation to the posterior, in this case to the density, and obtains an approximation of a similar form.  This approximation is also about a center other than the posterior expectation; in this case, the expansion is in the neighborhood of a true parameter value.
The notation here is similar to that of \citet{Hartigan:1965}.  Suppose data $X_1,\ldots,X_n$ is observed, with observations independent and identically distributed, conditional on a scalar parameter $\theta$, with common log density $g(x_{i}|\theta)$ and log prior density $h(\theta)$.  
Let $\omega_j=\int_\Theta \theta^j\exp(h(\theta)+\sum_{i=1}^n g(x_i|\theta)~d\theta$; these quantities depend on the data.
The posterior expectation is then $\omega_1/\omega_0$.
Take $s=2$; then Weng's approximation to the posterior distribution function is 
\begin{equation*}\Phi(\newtheta)+\phi(\newtheta)\{
c_1+ \newtheta c_2+
(\newtheta^2-1) c_3+
(\newtheta^3-\newtheta) c_4+
(\newtheta^5-10\newtheta^3+15\newtheta) c_6\},
\end{equation*}
and to the density is
\begin{align*}
w(\newtheta)=\phi(\newtheta)-\phi(\newtheta)\{ \newtheta c_1+ (\newtheta^2-1) c_2+
(\newtheta^3-\newtheta) c_3+ (\newtheta^4-6\newtheta^2+3) c_4 \\
+(\newtheta^6-15\newtheta^4+45\newtheta^2-15) c_6\},
\end{align*}
uniformly (in $\theta$) to $O(n^{-3/2})$.
Then the expectation associated with this density approximation is $-c_1$, the variance associated with this approximation is $1-c_1^2-2 c_2$, 
and the approximation to the density $d(\rho)$ of $\rho=(\theta-\hat\theta_n)/\hat\sigma-c_1$
satisfies 
\begin{equation}\label{corrden}
d(\rho)=w(\rho (1-c_1^2-2 c_2)^{1/2}-c_1) (1-c_1^2-2 c_2)^{1/2}+O(n^{-3/2}).\end{equation}
\citet{Weng:2010} notes that 
\begin{equation}\label{coeforder}
c_1, c_3=O(n^{-1/2}), c_2, c_4, c_6=O(n^{-1}).\end{equation}
Expanding $w(\rho (1-c_1^2-2 c_2)^{1/2}-c_1) (1-c_1^2-2 c_2)^{1/2}$ in terms of powers of $n^{1/2}$, and bounding errors using, for example, \citet[Theorem 2.5.3]{Kolassa:2006},  one can exhibit $w(\rho (1-c_1^2-2 c_2)^{1/2}-c_1) (1-c_1^2-2 c_2)^{1/2}$ of the form 
\begin{align}
\phi(\rho)-\phi(\rho)\{ \rho c^*_1+ (\rho^2-1) c^*_2+
(\rho^3-\rho) c^*_3+ (\rho^4-6\rho^2+3) c^*_4+\nonumber\\
(\rho^6-15\rho^4+45\rho^2-15) c^*_6\}\nonumber
\end{align}
for constants $c^*_j$ satisfying (\ref{coeforder}), and furthermore, $c^*_1=c^*_2=0$. Hence (\ref{corrden}) is an Edgeworth expansion to $O(n^{-3/2})$.

\begin{remark}
We have given two proofs of the validity of Edgeworth expansion of the posterior. The first is a direct proof for the formal Edgeworth expansion (Theorem~\ref{thm:main}), while the second is not a formal Edgeworth expansion, but rather shows how to correct Weng's expansion due to using the wrong center. As the above arguments demonstrate, one can obtain an Edgeworth-type expansion by centering at the maximum likelihood estimate $\hat{\theta}_{n}$ and correcting. Such an expansion could have the same form as an Edgeworth series, but would not be a formal Edgeworth expansion, and would require additional work to compute the correction factors.
\end{remark}
\begin{remark}
We have used the Laplace approximation of the cumulants only to show that they are of the correct asymptotic order. It is not necessary to use the Laplace approximation for implementation of the expansion. In practice, any sufficiently accurate estimator of the posterior moments could be used to implement the Edgeworth expansion. One approach is given by \citet{Hartigan:1965}. Another is to use the constants given by Weng and adjust them accordingly. 
\end{remark}
\begin{remark}
Weng's expansion is actually more similar to a Gram-Charlier expansion of the posterior, not an Edgeworth expansion. Weng ensures that the pseudo-moments are of the correct order, but not the cumulants. For example, in her displayed equation (45), the set $J_{2}$ (corresponding to the $n^{-1}$ term in the expansion) includes the sixth Hermite polynomial and its multiplier. For an Edgeworth expansion, this term is discarded.
\end{remark}

\section{Example}
\label{sec:example}

Consider a random variable having a binomial distribution, $X\sim \text{Bin}(\theta, n)$ with a beta prior, $\theta \sim \text{Beta}(a,b)$. Suppose that $a=0.5$, $b=4.0$, $n=5$ and $x=2$. This example was previously considered by \citet{Weng:2010}, who, using Stein's identity, derived an asymptotic expansion using up to 40 moments. This is an ideal example for illustrating the performance of posterior expansions, because the sample size is small, and the normal approximation is inaccurate, due to the skewness in the posterior.

\begin{figure}[htp] 
\caption{\label{fig:1}
Density approximation of Weng (2010).}
\centering
\includegraphics[width=0.7\textwidth]{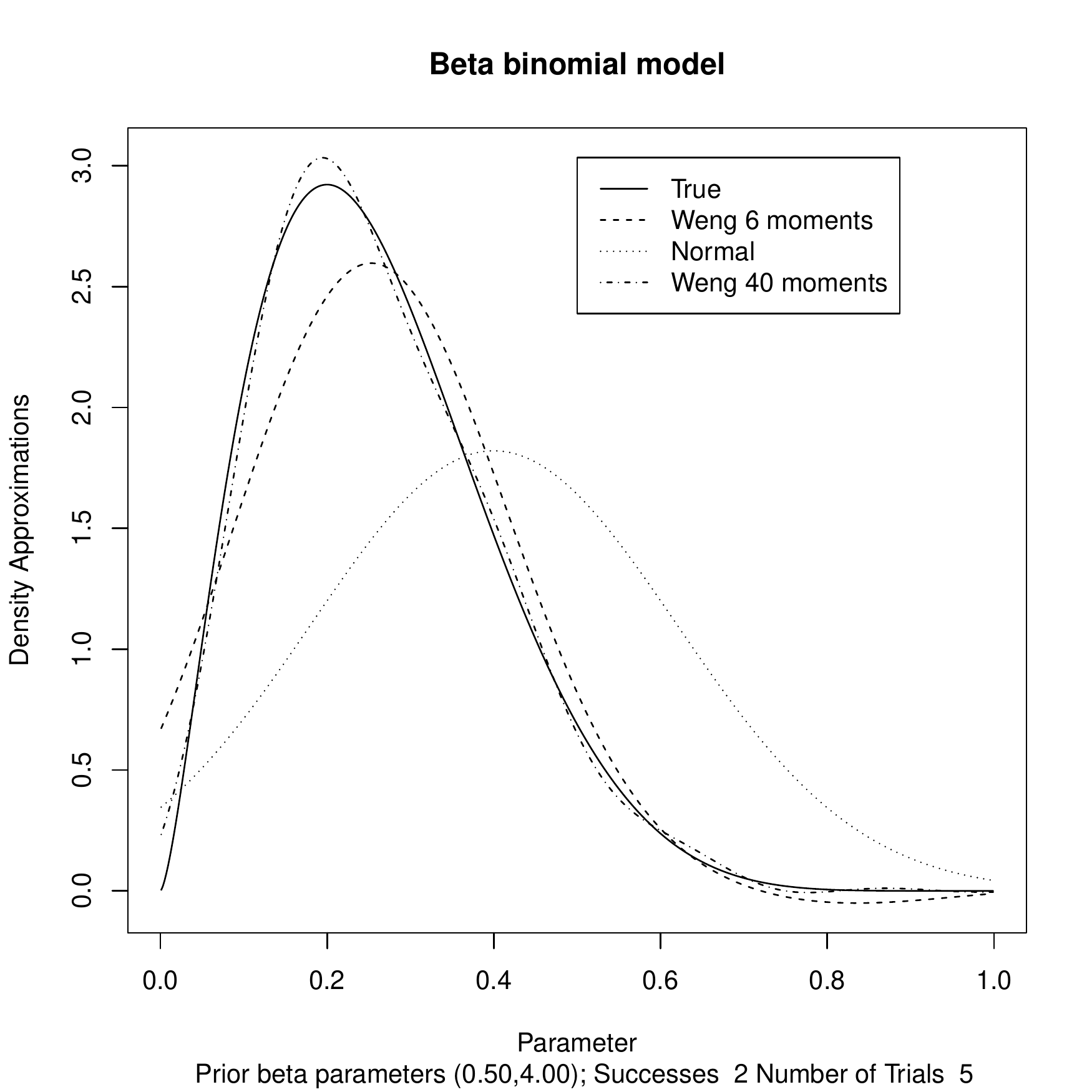}
\end{figure}

\begin{figure}[htp] 
\caption{\label{fig:2}
Edgeworth approximation of the posterior density.}
\centering
\includegraphics[width=0.7\textwidth]{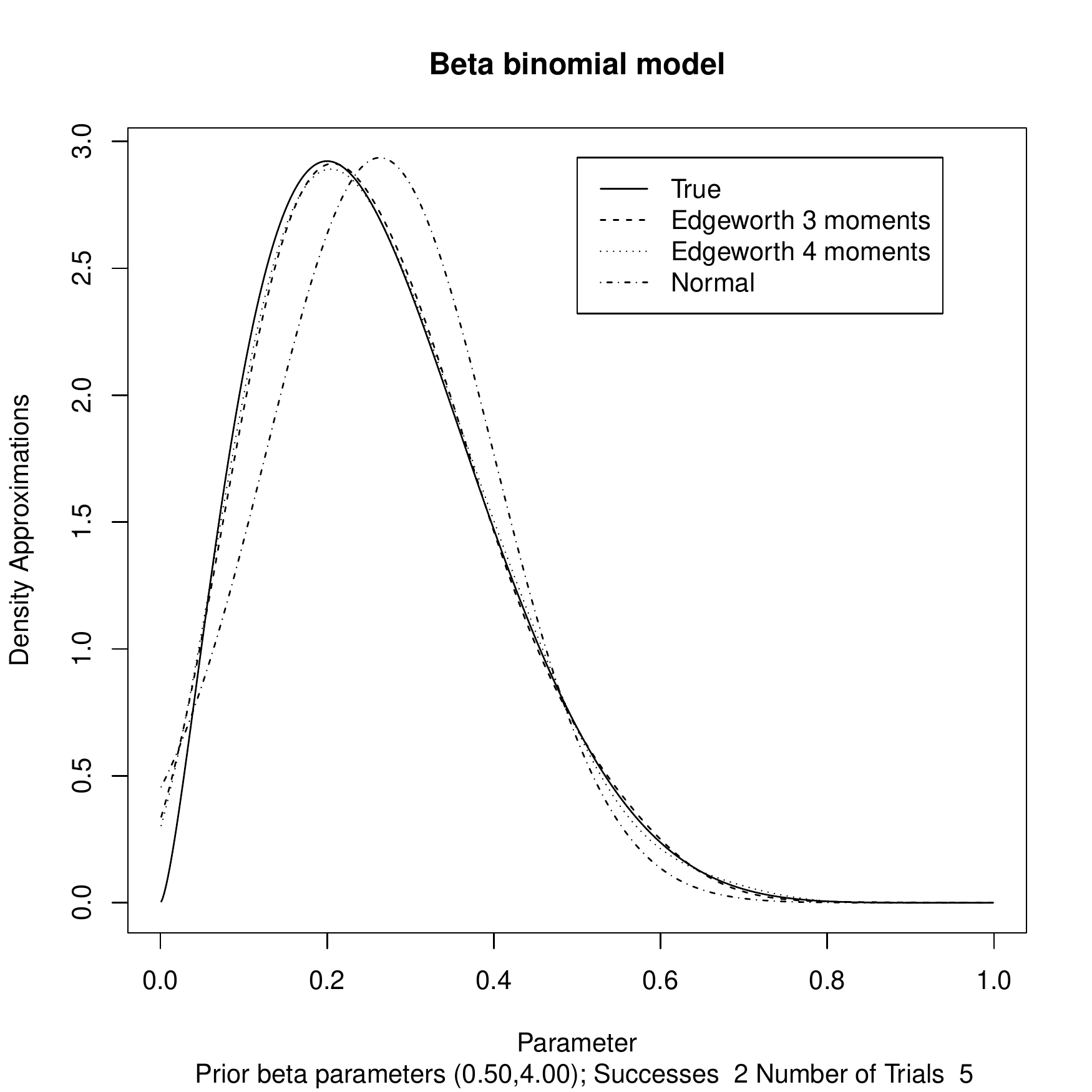}
\end{figure}

 Weng's expansion, for the density of $\theta$, is given in Figure~\ref{fig:1}. This figure should be compared with Figure~\ref{fig:2}, showing our posterior Edgeworth approximation for the posterior density of $\theta$. Note that the standard Edgeworth approximation, with four moments, behaves better than the Stein's identity approach using 40 moments.
To understand this phenomenon, consider the formal Edgeworth derivation due to \citet{Davis:1976} and presented by \citet{McCullagh:1987}.
When constructing an approximation of a density $f$ around a baseline density $g$, obtain the formal series
\[
f(x)=g(x)\sum_{j=0}^{\infty}h_{j}(x)\mu^*_{j}/j!,
\]
where the functions $h_{j}$ are ratios of derivatives of $g$ to $g$ itself, and the coefficients $\mu^*_{j}$ represent the results of calculating differences in cumulants between $f$ and $g$, and applying the standard relationship giving moments from cumulants to these cumulant differences to get pseudo-cumulants.  Standard Edgeworth approximation techniques and the method of \citet{Weng:2010} use as $g$ a normal density; standard Edgeworth approximations use $g$ with mean and variance matching $f$.  When applying Edgeworth techniques to a posterior, then, the standard approach is to match the mean and variance.  \citet{Weng:2010} uses instead the maximum likelihood estimate and its usual standard error, and so convergence is slower.

Figure~\ref{fig:3} displays the absolute error the normal approximation, and the Edgeworth approximation of the posterior density.

\begin{figure}[htp] 
\caption{\label{fig:3}
Absolute error of density approximations.}
\centering
\includegraphics[width=0.7\textwidth]{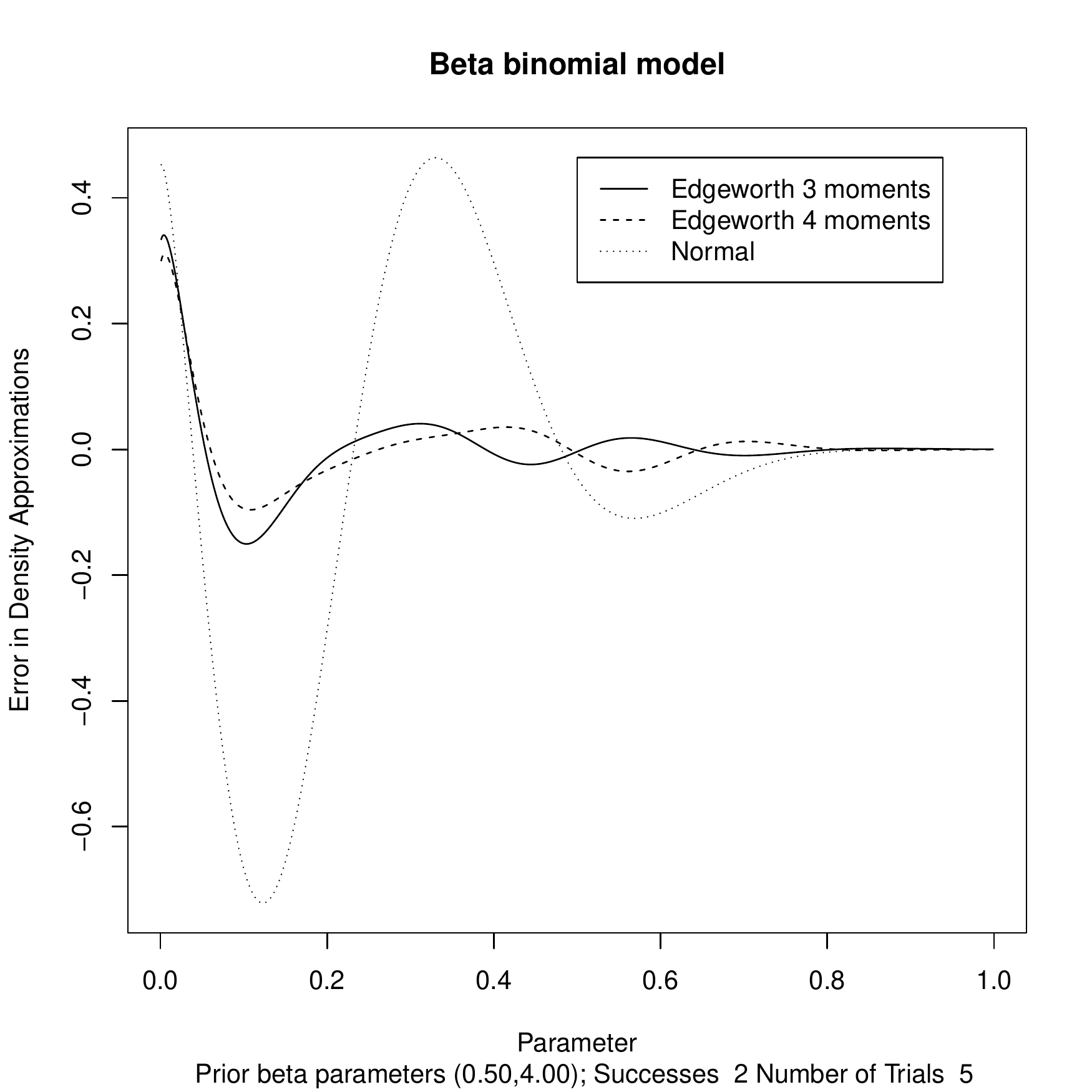}
\end{figure}

\section*{Acknowledgement}
Both authors were partially supported by the National Science Foundation. The second author is grateful to J.~K. Ghosh, Jens Jensen, Trevor Sweeting and Alastair Young for helpful correspondence and discussion related to this topic during 2010-2011.

\bibliographystyle{plainnat}
\bibliography{edgepostrefs}

\end{document}